\documentclass[12pt]{amsart}
\usepackage{amsmath}
\usepackage{amssymb}

\def\sideremark#1{\ifvmode\leavevmode\fi\vadjust{\vbox to0pt{\vss
  \hbox to 0pt{\hskip\hsize\hskip1em
  \vbox{\hsize2.5cm\tiny\raggedright\pretolerance10000
  \noindent #1\hfill}\hss}\vbox to8pt{\vfil}\vss}}}

\newtheorem{theorem}{Theorem}[section]

\newtheorem{corollary}[theorem]{Corollary}

\theoremstyle{definition}

\newtheorem{question}[theorem]{Question}
\newtheorem{conjecture}[theorem]{Conjecture}

\renewcommand{\H}{\mathbb{H}}

\renewcommand{\phi}{\varphi}
\newcommand{\R}{\mathbb{R}}
\newcommand{\iAH}{{\rm int}(AH(M))}
\newcommand{\rs}{\widehat{\mathbb{C}}}
\newcommand{\PSL}{{\mathbb{PSL}}}
\newcommand{\Isom}{{\rm Isom}}
\newcommand{\C}{{\mathbb{C}}}
\newcommand{\grad}{{\rm grad}}
\begin{document}

\title[Marden's Tameness Conjecture]{Marden's Tameness Conjecture:
History and Applications}

\author{Richard D. Canary}
\address{University of Michigan \\ Ann Arbor, MI 48109 \\ USA}
\thanks{research supported in part by grants from the National Science Foundation}

\begin{abstract}
Marden's Tameness Conjecture predicts that every hyperbolic 3-manifold
with finitely generated fundamental group is homeomorphic to the interior of
a compact 3-manifold. It was recently established by Agol and Calegari-Gabai.
We will survey the history of work on this conjecture and discuss its many applications.
\end{abstract}

\maketitle

\section{Introduction}

In a seminal paper, published in the Annals of Mathematics in 1974, Al Marden \cite{marden}
conjectured that every hyperbolic 3-manifold with finitely generated fundamental
group is homeomorphic to the interior of a compact 3-manifold. This conjecture
evolved into one of the central  conjectures in the theory of hyperbolic 3-manifolds.
For example, Marden's Tameness Conjecture
implies Ahlfors' Measure Conjecture (which we will discuss later). It is a crucial piece in
the recently completed classification of hyperbolic 3-manifolds with finitely generated fundamental
group. It also has
important applications to geometry and dynamics of hyperbolic 3-manifolds and gives important
group-theoretic information about fundamental groups of hyperbolic 3-manifolds.

There is a long history of partial results in the direction of Marden's Tameness Conjecture
and it was recently completely established by Agol \cite{agol} and Calegari-Gabai
\cite{calegari-gabai}. In this brief expository paper, we will survey the history of these
results and discuss some of the most important applications.

\medskip\noindent
{\bf Outline of paper:} In section 2, we recall basic definitions from the theory of
hyperbolic 3-manifolds. In section 3, we construct a 3-manifold with finitely
generated fundamental group which is not homeomorphic to the interior of a compact
3-manifold. In section 4, we discuss some of the historical background for Marden's
Conjecture and introduce the conjecture. In section 5, we introduce Thurston's notion
of geometric tameness, which turns out to be equivalent to topological tameness.
 In section 6, we sketch the history of partial
results on the conjecture. In section 7, we give geometric applications of the Tameness
Theorem, including Ahlfors' Measure Conjecture, spectral theory of hyperbolic 3-manifolds
and volumes of closed hyperbolic 3-manifolds. In section 8, we discuss group-theoretic
applications, including applications to the finitely generated intersection property and
separability properties of subgroups of Kleinian groups. In section 9, we discuss 
Simon's conjecture that the interior of a cover with finitely generated fundamental
group of a compact irreducible 3-manifold is topologically tame and give Long
and Reid's proof of Simon's conjecture from Simon's work and the Tameness Theorem.
In section 10, we explain the role of the Tameness Theorem in the classification of
hyperbolic 3-manifolds with finitely generated fundamental group. In section 11,
we discuss applications to the deformation theory of hyperbolic 3-manifolds.

\medskip\noindent
{\bf Acknowledgements:} I would like to thank Darren Long and Alan Reid for kindly
allowing me to include their proof of Simon's Conjecture from the Tameness Theorem
and to thank Alan Reid for allowing me to include his proof of Theorem \ref{reid}.
I would like to thank Yair Minsky and Alan Reid for helpful comments on earlier
versions of this manuscript. Finally, I would like to thank the organizers for their
patience and for an enjoyable and informative conference.

\section{Basic Definitions}

A (complete) hyperbolic 3-manifold is a complete Riemannian 3-manifold with  constant
sectional curvature -1. {\bf Throughout this paper we will assume that all manifolds
are orientable.} Any hyperbolic 3-manifold may be obtained as the quotient $N=\H^3/\Gamma$
where $\Gamma$ is a group of orientation-preserving isometries acting properly discontinuously
on $\H^3$. The group $\Gamma$ is called a {\em Kleinian group}. (More generally, a Kleinian
group is a discrete subgroup of $\Isom_+(\H^3)$. In this paper, all Kleinian groups will
be assumed to be torsion-free, so that their quotient is a hyperbolic 3-manifold.)

The group $\Isom_+(\H^3)$ of orientation-preserving isometries of $\H^3$ is naturally
identified with the group $\PSL_2(\C)$ of Mobius transformations of the Riemann sphere
$\rs$, which we regard as the boundary at infinity of $\H^3$. So, if $N=\H^3/\Gamma$, then
$ \Gamma$ acts also as a group of conformal automorphisms of $\rs$. We divide $\rs$ up
into the {\em  domain of discontinuity} $\Omega(\Gamma)$ which is the largest open
subset of $\rs$ on which $\Gamma$ acts properly discontinuously, and its complement,
 $\Lambda(\Gamma)$,which is called
the {\em limit set}.

Since $\Gamma$ acts properly discontinuously on $\Omega(\Gamma)$, the quotient
$\partial_c(N)=\Omega(\Gamma)/\Gamma$ is a Riemann surface, called the {\em conformal
boundary}. One may naturally append $\partial_c(N)$ to $N$ to obtain a 3-manifold with
boundary 
$$\hat N=N\cup \partial_c(N)=(\H^3\cup \Omega(\Gamma))/\Gamma.$$ 

Of course, the nicest possible situation is that this bordification $\hat N$ gives a compactification
of $N$, and $N$ will be called {\em convex cocompact} if $\hat N$ is compact. More generally,
$N$ is said to be {\em geometrically finite} if $\hat N$ is homeomorphic to $M-P$ where
$M$ is a compact 3-manifold and $P$ is a finite collection of disjoint annuli and tori
in $\partial M$. (These definitions are non-classical, see Marden \cite{marden} and
Bowditch \cite{bowditch} for a discussion of their equivalence to more standard definitions.)
If $\Gamma$ contains no parabolic elements, equivalently if every homotopically non-trivial
simple closed curve in $N$ is homotopic to a closed geodesic, then $N$ is said to have
{\em no cusps}. If $N$ has no cusps, then it is geometrically finite if and only if it is
convex cocompact.  If $N=\H^3/\Gamma$ is geometrically finite, we will also say
that the associated Kleinian group $\Gamma$ is geometrically finite.

We will say that a 3-manifold is {\em topologically tame} if it is homeomorphic to
the interior of a compact 3-manifold. It is clear from the definition we gave of geometric
finiteness, that geometrically finite hyperbolic 3-manifolds are topologically tame.

If $N$ is a hyperbolic 3-manifold, then its {\em convex core} $C(N)$ is the smallest
convex submanifold $C$  of $N$ such that the inclusion of $C$ into $N$ is a homotopy
equivalence. More concretely, $C(N)$ is the quotient $CH(\Lambda(\Gamma))/\Gamma$
of the convex hull of the limit set. Except in the special case where $\Lambda(\Gamma)$
is contained in a circle in the Riemann sphere, the convex core $C(N)$ is homeomorphic
to $\hat N$. In particular, $N$ is convex cocompact if and only if its convex core is compact.

Thurston \cite{thurston-notes} showed that the boundary of the convex core is a hyperbolic surface, in its
intrinsic metric. There is a strong relationship between the geometry of the boundary of
the convex core and the geometry of the conformal boundary. We recall that the conformal
boundary admits a unique hyperbolic metric in its conformal class, called the Poincar\'e metric.
Sullivan showed that  there exists a constant $K$ such that if $N$ has finitely generated,
freely indecomposable fundamental group, then there exists a $K$-bilipschitz map between $\partial C(N)$ and
$\partial_c(N)$. Epstein and Marden \cite{epstein-marden} gave a careful proof of this
result and showed that $K\le 82.8$. More complicated analogues of Sullivan's result
are known to hold when the fundamental group of $N$ is not freely indecomposable,
see Bridgeman-Canary \cite{bridgeman-canary}.

\section{Topologically wild manifolds}

In order to appreciate Marden's Tameness Conjecture we will sketch a  construction of a 3-manifold
with finitely generated fundamental group which is not topologically tame.

Whitehead \cite{whitehead} gave an example of a simply connected  (in fact, contractible)
3-manifold which is not homeomorphic to the interior of a ball. One may  use
Whitehead's construction to give examples of 3-manifolds with non-trivial finitely generated
fundamental group which are not topologically tame.  Scott and Tucker \cite{scott-tucker}
give an interesting collection of examples of topologically wild 3-manifolds with
a variety of properties.

The following example is essentially drawn from Tucker \cite{tucker}.
Whitehead's example is given as a nested union of solid tori.  The fundamental group
of our example will be the free group of rank two and will be obtained as a nested
union of handlebodies of genus 2.

Let $H$ denote a handlebody of genus 2
and let $g:H\to H$ be an embedding of $H$ into the interior of $H$ such that
\begin{enumerate}
\item
$g_*:\pi_1(H) \to \pi_1(H)$ is the identity map,
\item
$\partial H$ and $g(\partial H)$ are both incompressible  in
$H-g(H)$ (i.e. the inclusions induce injections
of $\pi_1(\partial H)$ and $\pi_1(g(\partial H))$ into $\pi_1(H-g(H))$)
\item
$\pi_1(g(\partial  H))$ and $\pi_1(\partial H)$
give rise to proper subgroups of $\pi_1(H-g(H))$.
\end{enumerate}
Let $B=H-g(H)$.

We then construct a manifold by letting $H_0=H$  and $H_\infty=\bigcup H_n$
where the pair $(H_n,H_{n-1})$ is homeomorphic to $(H,g(H))$ for all $n\ge 1$.
More concretely, we could define $H_n=H_{n-1}\cup B$
where we think of  $\partial H_{n-1}$ as the copy of $\partial H$ in the
$(n-1)^{\rm st}$ copy of $B$ and identify $\partial H$ with $g(\partial H)$
by the map $g$. Notice that $H_n$ is a handlebody for all $n$.

The inclusion of $H_0$ into $H_\infty$ is a homotopy
equivalence, since the inclusion of $H_{n-1}$ into  $H_n$ is a
homotopy equivalence for all $n$.
Therefore, $\pi_1(H_\infty)$ is the free group on
two generators. Note also that each $H_n$ is a compact core for $H_\infty$.
However, one can repeatedly apply the Seifert-Van Kampen theorem to
show that $H_\infty-H_0$ has infinitely generated fundamental group.
It then follows that $H_\infty$ is not topologically tame, since the complement
of a compact submanifold of a topologically tame 3-manifold must have finitely
generated fundamental group.

To check that $\pi_1(H_\infty-H_0)$ is infinitely generated, we observe
that $\pi_1(H_n-H_0)=\pi_1(H_{n-1}-H_0)*_{\pi_1({\partial H})}\pi_1(B)$ and that
$\pi_1(\partial H)$ injects into both factors. Moreover, $\pi_1(\partial H)$ is also
a proper subgroup of each factor. It follows that $\pi_1(H_{n-1}-H_0)$ injects
into $\pi_1 (H_n-H_0)$ and that  its image is a proper subgroup.
If $H_\infty-H_0$ had finitely generated fundamental group, then its generators
would have representatives lying in $H_{n-1}-H_0$ for some $n$ which would
contradict the previous sentence. Alternatively, one could note that
$$\pi_1(H_\infty-H_0)=\pi_1(B)*_{\pi_1(\partial H)}\pi_1(B) *_{\pi_1(\partial H)}\pi_1(B)*_{\pi_1(\partial H)}\cdots$$ 
and use properties of amalgamated free products to check that this group is infinitely
generated.

Tucker \cite{tucker} showed that one can always detect that an irreducible 3-manifold
is topologically wild by considering the fundamental group of the complement of
a compact submanifold.

\begin{theorem}{\rm (Tucker \cite{tucker})}
An irreducible 3-manifold $M$ (without boundary) is not topologically tame if and only if there exists
a compact submanifold $C$  of $M$ such that $\pi_1(M-C)$ is infinitely generated.
\end{theorem}

\section{Prehistory and the conjecture}

Hyperbolic 2-manifolds with finitely generated fundamental group are well-known
to be geometrically finite. For many years, it was unclear whether the analogous
statement was true for hyperbolic 3-manifolds.
In 1966, Leon Greenberg \cite{greenberg} established the existence of hyperbolic 3-manifolds
with finitely generated fundamental group which are not geometrically finite. Troels J\o rgensen
\cite{jorgensen}
was the first to explicitly exhibit geometrically infinite hyperbolic 3-manifolds with
finitely generated fundamental group.

One early piece of evidence for Marden's Tameness Conjecture was provided by Ahlfors'
Finiteness Theorem \cite{ahlfors-AFT} which asserts that if $\Gamma$ is finitely
generated, then $\partial_c(N)$ is a finite collection of finite type Riemann surfaces.
Alternatively, one could say that the conformal boundary has finite area in its  Poincar\'e
metric.

On the topological side, Peter Scott \cite{scott}  showed that any 3-manifold  $M$ with finitely
generated fundamental group contains a {\em compact core}, i.e. a compact submanifold $C$
such that the inclusion of $C$ into $M$ induces an isomorphism from $\pi_1(C)$ to
$\pi_1(M)$. If $M$ is a hyperbolic 3-manifold, or more generally if $M$ is irreducible, one
may assume that the inclusion of $C$ into $M$ is a homotopy equivalence. This implies,
in particular, that finitely generated 3-manifold groups are actually finitely presented.
It follows from work of McCullough, Miller and Swarup \cite{MMS} that if $M$ is  irreducible and topologically
tame, then it is homeomorphic to the interior of its compact core.

In 1974, Al Marden published a long paper \cite{marden} which was the
first papers to bring to bear the classical results of 3-manifold topology on
the study of hyperbolic 3-manifolds. Previously, most of the work on hyperbolic
3-manifolds, was done by considering the actions of their fundamental groups on
the Riemann sphere. In the appendix of  this paper, Marden asked two  prescient questions
which we will rephrase in our language.

\medskip\noindent
{\bf Marden's first question:} {\em If $N$ is a hyperbolic 3-manifold with finitely generated fundamental
group, is $N$ topologically tame?}

\medskip\noindent
{\bf Marden's second question:} {\em Is there a necessary and sufficient condition which
guarantees that a compact 3-manifold $M$ is hyperbolizable? (A compact manifold is hyperbolizable
if there admits a complete
hyperbolic metric on the interior of $M $.)}

\medskip

The first question became known as Marden's Tameness Conjecture, while
the second question foreshadows Thurston's Geometrization Conjecture.
In particular, Thurston's Geometrization Conjecture predicted that a compact
3-manifold is hyperbolizable if it is irreducible, atoroidal and its fundamental
group is not virtually abelian. We recall that a compact 3-manifold $M$ is {\em irreducible} if every embedded 2-sphere
bounds a ball and is called {\em atoroidal} if $\pi_1(M)$ does not contain a free abelian
subgroup of rank two. Thurston established his hyperbolization  conjecture whenever $M$
is Haken, e.g. whenever $M$ has non-empty boundary.
See Morgan \cite{morgan} or Kapovich \cite{kapovich-book} for extensive discussions
of Thurston's Hyperbolization  Theorem. Perelman  \cite{perelman,perelman2} recently gave a proof of Thurston's
entire Geometrization Conjecture. (See Kleiner-Lott \cite{kleiner-lott}, Morgan-Tian \cite{morgan-tian}
and Cao-Zhu \cite{cao-zhu} for expositions of Perelman's work.)

\section{Geometric tameness}

In Thurston's work on the Geometrization Conjecture he developed a notion called
geometric tameness. His definition only worked in the setting of hyperbolic 3-manifolds
with freely indecomposable fundamental group. We will give a definition of geometric
tameness, first introduced in \cite{canary-ends}, which works in a more general setting.

In order to motivate the definition we will consider a specific example of a geometrically
infinite manifold. Thurston \cite{thurston2} showed that any atoroidal 3-manifold
which fibers over the circle, whose fundamental group is not virtually abelian, 
is hyperbolizable. (Thurston was inspired by J\o rgensen's example \cite{jorgensen} which
was the cover of a hyperbolic 3-manifold which fibers over the circle.)
Let $M$ be a closed hyperbolic 3-manifold
which fibers over the circle with fiber the closed surface $S$ and let $\hat M$ be the regular
cover associated to $\pi_1(S)$. The fibered manifold $M$ is obtained from $S\times [0,1]$ by
gluing $S\times \{0\}$ to $S\times\{1\}$ by a homeomorphism $\phi:S\to S$
and the manifold $\hat M$ is
obtained from infinitely many copies of $S\times [0,1]$ stacked one on top of the other.
So, $\hat M$ is homeomorphic to $S\times \R$ and the group of covering transformations
is generated by $\hat \phi:S\times \R \to S\times \R$ where $\hat\phi(x,t)=(\phi(x),t+1)$.
One sees that the ends of $\hat M$ grow linearly, which is surprising for a hyperbolic
manifold. In particular, if one considers a minimal surface in $M$ in the homotopy class of 
$S$, its pre-image in $\hat M$ is an infinite family of surfaces exiting both ends of
$\hat M$. The key property of these surfaces is that they have curvature $\le -1$.
Thurston realized that the existence of such a family of surfaces was both
widespread and quite useful.

In our definition of geometric tameness, we will make use of simplicial hyperbolic
surfaces. We will give a careful definition of a simplicial hyperbolic surface $f:S\to N$, but we first
note that the key issue is that the induced metric  on $S$ has curvature $\le -1$ (in the sense of
Alexandrov.) Thurston \cite{thurston-notes} originally made use of pleated surfaces,
Minsky \cite{minsky-harmonic} showed that one can use harmonic maps, Bonahon \cite{bonahon}
pioneered the use of simplicial hyperbolic surfaces, Calegari and Gabai \cite{calegari-gabai}
used shrinkwrapped surfaces and Soma \cite{soma} used ruled wrappings (which are a simplicial
analogue of  shrinkwrapped surfaces.)

A map $f:S\to N$ is a {\em simplicial hyperbolic surface} if there exists a triangulation
$T$ of $S$ such that $f$ maps faces of $T$ to totally geodesic immersed triangles in $N$ and
the total angle of the triangles about any vertex adds up to at least $2\pi$. We note
that we allow our triangulations to have the property that vertices and edges of faces
may be identified. For example, one can obtain a triangulation of a torus with two faces by adding
the diagonal to the usual square gluing diagram. Simplicial hyperbolic surfaces share many
useful properties with actual hyperbolic surfaces. For example, their area is bounded
above by $2\pi|\chi(S)|$ and the diameter of each component of their ``thick part'' is uniformly bounded.

Let $N$ be a hyperbolic 3-manifold with finitely generated fundamental group.
For the purposes of simplifying the definition we will assume that $N$ has no cusps.
Let $C$ be a compact core for $N$. The ends of $N$ may be identified with
the components of $N-C$. For our purposes, a neighborhood  $U$ of an end $E$ is
a subset of $E$ such that $E-U$ has compact closure. We will say that an end is
{\em geometrically finite} if it has a neighborhood disjoint from the convex core. 
We say that an end is {\em simply degenerate} if it has a neighborhood $U$ which is
homeomorphic to $S\times (0,\infty)$ for some closed surface $S$ and there exists a
sequence $\{ f_n:S\to U\}$ of simplicial hyperbolic surfaces, such that 
\begin{enumerate}
\item
given any compact subset
$K$ of $N$, $f_n(S)\cap K$ is empty for all but finitely many $n$, and
\item
for all $n$,
$f_n$ is homotopic, within $U$, to the map $h_1:S\to U$ given by $h_1(x)=(x,1)$.
\end{enumerate}
A hyperbolic 3-manifold with finitely generated fundamental group is said to
be {\em geometrically tame} if each of its ends is either geometrically finite
or simply degenerate.

It is not difficult to show that any geometrically finite end has a neighborhood
homeomorphic to $S\times (0,\infty)$ for some closed surface $S$.
Therefore, one easily observes that geometrically tame hyperbolic 3-manifolds
are topologically tame. Thurston originally gave a much weaker definition of
geometric tameness in the setting of hyperbolic 3-manifolds with freely indecomposable
fundamental group. In this setting, Thurston \cite{thurston-notes} was able to show that
his weaker definition of geometric tameness implied topological tameness and
in fact implied the definition given here.

\section{History}

In this section, we will give a brief history of the partial results on Marden's Tameness
Conjecture leading up to its final solution. Many of these results involved looking at
limits of geometrically finite or topologically tame hyperbolic 3-manifolds.  Readers
who are not interested in the historical development may prefer to skip ahead to the next
section.

There are two types of convergence, algebraic and geometric, which play a prominent role in
the theory of hyperbolic 3-manifolds.
A sequence $\{\rho_n:G\to \PSL_2(\C)\}$  of discrete faithful representations is
said to {\em converge algebraically} to $\rho:G\to \PSL_2(\C)$ if it converges
in the compact-open topology on ${\rm Hom}(G,\PSL_2(\C))$. If $G$ is not
virtually abelian (i.e. does not contain a finite index abelian subgroup), then
$\rho$ is also discrete and faithful (see Chuckrow \cite{chuckrow} and 
J\o rgensen \cite{jorgensen-ineq}.)
A sequence of Kleinian groups $\{\Gamma_n\}$ is said to {\em converge geometrically}
to a Kleinian group $\Gamma$ if every $\gamma\in\Gamma$ arises as a limit
of a sequence $\{\gamma_n\in\Gamma_n\}$ and every limit $\beta$ of a sequence
of elements $\{\gamma_{n_k}\}$ in a subsequence $\{\Gamma_{n_k}\}$ of $\{\Gamma_n\}$
lies in $\Gamma$. We say that a sequence of hyperbolic 3-manifolds $\{ N_n\}$
converges geometrically to $N$ if one may write $N_n=\H^3/\Gamma_n$ and $N=\H^3/\Gamma$
so that $\{\Gamma_n\}$ converges geometrically to $\Gamma$. This implies,
see \cite{CEG} or \cite{BP}, that $\{ N_n\}$ converges in the sense of Gromov to $N$, i.e.
 ``larger and larger'' subsets of $N_n$ look
``increasingly like'' large subsets of $N$ (as $n$ goes to $\infty$.) A sequence
$\{\rho_n:G\to \PSL_2(\C)\}$  of discrete faithful representations is
said to {\em converge strongly} to $\rho:G\to \PSL_2(\C)$  it it converges algebraically
and $\{\rho_n(G)\}$ converges geometrically to $\rho(G)$. 

Thurston \cite{thurston-notes} proved the following theorem concerning limits of geometrically finite
hyperbolic 3-manifolds with freely indecomposable fundamental group. He used his
theorem in the proof of his geometrization theorem.

\begin{theorem}{\rm (Thurston \cite{thurston-notes})}
Let $G$ be a finitely generated, torsion-free, freely indecomposable, non-abelian group.
If $\{\rho_n:G\to \PSL_2(\C)\}$ is a sequence of discrete, faithful, geometrically tame representations converging
algebraically to $\rho:G\to\PSL_2(\C)$
such that $\rho_n(g)$ is parabolic for some $n$ if and only if $\rho(g)$ is parabolic,
then $\{\rho_n\}$ converges strongly to $\rho$ and $\rho(G)$ is geometrically tame.
\end{theorem}

In a breakthrough paper, Bonahon \cite{bonahon}  proved that all hyperbolic 3-manifolds with
freely indecomposable, finitely generated fundamental group are geometrically tame,
and hence topologically tame.

\begin{theorem}{\rm (Bonahon\cite{bonahon})}
If $N$ is a hyperbolic 3-manifold with finitely generated, freely indecomposable group
then $N$ is geometrically tame.
\end{theorem}

Canary \cite{canary-ends} used Bonahon's work to show that topological tameness
and geometric tameness are equivalent notions.

\begin{theorem}{\rm (Canary \cite{canary-ends})}
Let $N$ be a hyperbolic 3-manifold with finitely generated fundamental group.
Then, $N$ is topologically tame if and only if $N$ is geometrically tame.
\end{theorem}

Canary-Minsky \cite{canary-minsky} and Ohshika \cite{ohshika-limits} showed that,
in the absence of cusps,  strong limits of topologically tame hyperbolic 3-manifolds are themselves
topologically tame.
In order to more easily state this result, we define a representation $\rho:G\to\PSL_2(\C)$ to be
{\em purely hyperbolic} if  $\rho(G)$ contains no parabolic elements.

\begin{theorem}{\rm (Canary-Minsky \cite{canary-minsky},Ohshika\cite{ohshika-limits})}
Let $G$ be a finitely generated, torsion-free, non-abelian group.
If $\{\rho_n:G\to \PSL_2(\C)\}$ is a sequence of discrete, faithful, topologically tame, purely
hyperbolic representations
converging strongly to a purely hyperbolic representation $\rho:G\to\PSL_2(\C)$,
then $\rho(G) $ is topologically tame.
\end{theorem}

By applying results of Anderson-Canary \cite{AClimits} or Ohshika \cite{ohshika-limits}
one can often guarantee that when both the approximates and the limit in an algebraically
convergent sequence are purely hyperbolic, then the sequence converges strongly. 

\begin{corollary}
Let $G$ be a finitely generated, torsion-free, non-abelian group.
If $\{\rho_n:G\to \PSL_2(\C)\}$ is a sequence of discrete, faithful, topologically tame, purely hyperbolic representations converging algebraically to a purely hyperbolic representation $\rho:G\to\PSL_2(\C)$
such that either $\Lambda(\rho(G))=\rs$ or $G$ is not a free product of surface groups,
then $\rho(G)$ is topologically tame.
\end{corollary}

Evans \cite{evans} was able to significantly weaken the assumption that the representations
are purely hyperbolic.

\begin{theorem}{\rm (Evans\cite{evans})}
Let $G$ be a finitely generated, torsion-free, non-abelian group.
If $\{\rho_n:G\to \PSL_2(\C)\}$ is a sequence of discrete, faithful, topologically tame representations converging algebraically to $\rho:G\to\PSL_2(\C)$
such that 
\begin{enumerate}
\item
either $\Lambda(\rho(G))=\rs$ or $G$ is not a free product of surface groups, and
\item
if $g\in G$ and $\rho(g)$ is parabolic, then $\rho_n(g)$ is parabolic for all $n$,
\end{enumerate}
then $\rho(G)$ is topologically tame.
\end{theorem}

Juan Souto \cite{souto-tame} showed that if $N$ can be exhausted by compact cores then it is topologically
tame. (In fact, he proves somewhat more, but we will just state the simpler version.) Kleineidam
and Souto \cite{kleineidam-souto-tame} used Souto's result to show that if a
Masur domain lamination on a boundary component is not realizable,
then the corresponding end is tame (see \cite{kleineidam-souto-tame} for definitions). Souto's
work was quite influential in the later solutions of Marden's Tameness Conjecture.

\begin{theorem}{\rm (Souto\cite{souto-tame})}
If $N$ is a hyperbolic 3-manifold with finitely generated fundamental group
and $N=\bigcup C_i$ where $C_i$ is a compact core for $N$ and
$C_i\subset C_{i+1}$ for all $i$, then $N$ is topologically tame
\end{theorem}

Brock, Bromberg, Evans and Souto \cite{BBES} were able to show that ``most''
algebraic limits of geometrically finite hyperbolic 3-manifolds are topologically tame.

\begin{theorem}{\rm (Brock-Bromberg-Evans-Souto \cite{BBES})}
Let $G$ be a finitely generated, torsion-free, non-abelian group.
If $\{\rho_n:G\to \PSL_2(\C)\}$ is a sequence of discrete, faithful, geometrically finite representations converging algebraically to $\rho:G\to\PSL_2(\C)$
such that either
\begin{enumerate}
\item
$\Lambda(\rho(G))=\rs$,
 \item
 $G$ is not a free product of surface groups, or
\item
$\{\rho_n(G)\}$ converges geometrically to $\rho(G)$,
\end{enumerate}
then $\rho(G)$ is topologically tame.
\end{theorem}

Brock and Souto \cite{brock-souto} were able to resolve the remaining cases to prove that all
algebraic limits of geometrically finite hyperbolic 3-manifolds are topologically tame.

\begin{theorem}{\rm (Brock-Souto\cite{brock-souto})}
Let $G$ be a finitely generated, torsion-free, non-abelian group.
If $\{\rho_n:G\to \PSL_2(\C)\}$ is a sequence of discrete, faithful, geometrically finite representations converging algebraically to $\rho:G\to\PSL_2(\C)$,
then $\rho(G)$ is topologically tame.
\end{theorem}

In 2004, Ian Agol and the team of Danny Calegari and David Gabai announced proofs
of Marden's Tameness Conjecture.

\medskip\noindent
{\bf Tameness Theorem:}
{\rm (Agol\cite{agol},Calegari-Gabai\cite{calegari-gabai})}
{\em If $N$ is a hyperbolic 3-manifold with finitely generated fundamental group,
then $N$ is topologically tame.}

\medskip

Soma \cite{soma} later simplified the proof by combining ideas of Agol
and Calegari-Gabai. Bowditch \cite{bowditch-tame} also gives an account of the
result using ideas of Agol, Calegari-Gabai and Soma. He also describes how to
generalize the proof for manifolds of pinched negative curvature and uses it to
prove an analogue of Ahlfors' Finiteness Theorem in this setting. Choi \cite{choi} has an alternate approach to the proof.

\section{Geometric Applications}

In the remaining sections we will collect some of the major applications
and consequences of the solution of Marden's Tameness Conjecture.

Geometric tameness gives one strong control on the geometry of ends.
One manifestation of this control is a minimum principle for superharmonic
functions which Thurston \cite{thurston-notes} established for geometrically tame
hyperbolic 3-manifolds with incompressible boundary and Canary \cite{canary-ends}
generalized to the setting of topologically tame hyperbolic 3-manifolds. Given
the Tameness Theorem one need only require that our
manifold have finitely generated fundamental group.

\begin{theorem}{(Thurston \cite{thurston-notes},Canary\cite{canary-ends})}
Let $N$ be a hyperbolic 3-manifold with finitely generated fundamental group.
If $h:N\to (0,\infty)$ is a positive superharmonic function, i.e. ${\rm div}(\grad \ h)\ge0$,
then 
$${\rm inf}_{C(N)} h={\rm inf}_{\partial C(N)} h.$$
In particular, if $C(N)=N$, then $h$ is constant.
\end{theorem}

The main idea of the proof here is to consider the flow generated by $-\grad \ h$, i.e. the
flow in the direction of maximal decrease. The fact that $h$ is superharmonic
guarantees that this flow is volume non-decreasing. The fact that $h$ is positive guarantees
that the flow moves more and more slowly as one progress. Neighborhoods of radius
one of our simplicial hyperbolic surfaces have bounded volume, so act as narrows for
the flow. It follows that almost every flow line starting in $C(N)$ must exit the convex
core through its boundary, rather than flowing out one of the simply degenerate ends.

One of the most important applications of this minimum principle is a solution of
Ahlfors' Measure Conjecture. Ahlfors \cite{ahlfors-AMC} proved that a limit set
of a geometrically finite hyperbolic 3-manifold either has measure zero or
is the entire Riemann sphere.  He conjectured that this would hold for all
hyperbolic 3-manifolds with finitely generated fundamental group.

\begin{corollary} If $N=\H^3/\Gamma$ is a hyperbolic 3-manifold with finitely
generated fundamental group, then either $\Lambda(\Gamma)$ has measure
zero in $\rs$ or $\Lambda(\Gamma)=\rs$. Moreover, if $\Lambda(\Gamma)=\rs$
then $\Gamma$ acts ergodically on $\rs$, i.e. if $A\subset\rs$ is measurable and
$\Gamma$-invariant, then $A$ has either measure zero or full measure.
\end{corollary}

The proof of this corollary follows the same outline as Ahlfors' original proof.
We suppose that $\Lambda(\Gamma)\ne\rs$ and has positive measure and consider the harmonic
function $\tilde h:\H^3\to (0,1)$ given by letting $\tilde h(x)$ be the proportion of geodesic
rays emanating from $x$ which end at points in the limit set $\Lambda(\Gamma)$.
The function $\tilde h$ is $\Gamma$-invariant, so descends to a harmonic function
$h:N\to (0,1)$.
It is clear that $h\le {1\over 2}$ on $N-C(N)$, so the minimum principle applied
to $1-h$ implies that $h\le {1\over 2}$ on all of $N$. Therefore, $\tilde h\le {1\over 2}$ on $\H^3$.
On the other hand, as $x$ approaches a point of density of $\Lambda(\Gamma)$
along a geodesic, it is clear that $h(x)$ must approach 1, so we have achieved a contradiction.
(To establish the ergodicity of the action in the case that $\Lambda(\Gamma)=\rs$, one assumes
that $A$ is a $\Gamma$-invariant set which has neither full or zero measure and studies
the function $\tilde h$ which is the proportion of rays emanating from a point which
end in $A$.)

Another immediate consequence of this minimum principle is a characterization of
which hyperbolic 3-manifolds admit non-constant positive superharmonic functions.

\begin{corollary}
Let $N=\H^3/\Gamma$ be a hyperbolic 3-manifold with finitely generated fundamental
group. The manifold $N$ is strongly parabolic  (i.e. admits no non-constant positive
superharmonic functions) if and only if $\Lambda(\Gamma)=\rs$.
\end{corollary}

Sullivan \cite{sullivan-flow} showed that the geodesic flow of $N$ is ergodic if and only if it admits a (positive)
Green's function, so one can also completely characterize when
the geodesic flow of $N$ is ergodic.

\begin{corollary}
Let $N=\H^3/\Gamma$ be a hyperbolic 3-manifold with finitely generated fundamental
group. The geodesic flow of $N$ is ergodic  if and only if $\Lambda(\Gamma)=\rs$.
\end{corollary}

\medskip

Another collection of geometric applications of topological tameness involve the Hausdorff
dimension of the limit set and the bottom of the spectrum of the Laplacian. Patterson
\cite{patterson} and Sullivan \cite{sullivan-measure} showed that there are deep
relationships between these two quantities. In particular, they showed
that if $N=\H^3/\Gamma$ is geometrically finite, then
$$\lambda_0(N)=D(\Lambda(\Gamma))(2- D(\Lambda(\Gamma))$$
unless $D(\Lambda(\Gamma)<1$ in which case $\lambda_0(N)=1$.
Here, $D(\Lambda(\Gamma))$ denotes the Hausdorff dimension of the limit
set and $\lambda_0(N)=\inf {\rm spec}(-{\rm div}(\grad))$ is the bottom of the
spectrum of the Laplacian.

Sullivan \cite{sullivan-measure} and Tukia \cite{tukia}  showed that if $N$ is geometrically finite and
has infinite volume,
then $\lambda_0(N)>0$. Canary \cite{canary-laplacian} proved that if $N$ is topologically tame and geometrically
infinite, then $\lambda_0(N)=0$. (One does this by simply using the simplicial hyperbolic
surfaces exiting the end to show that the Cheeger constant of a geometrically infinite
manifold is 0.)

\begin{theorem}{\rm (Sullivan \cite{sullivan-measure},Tukia\cite{tukia},Canary\cite{canary-laplacian})}
Let $N=\H^3/\Gamma$ be a hyperbolic 3-manifold with finitely generated fundamental
group. Then $\lambda_0(N)=0$ if and only if either $N$ has finite volume or is geometrically
infinite.
\end{theorem}

Bishop and Jones \cite{bishop-jones} showed that geometrically infinite hyperbolic
3-manifolds have limit sets of Hausdorff dimension 2 without making use of
tameness.  Combining all the results we have mentioned one
gets the following result.

\begin{corollary}
Let $N=\H^3/\Gamma$ be a hyperbolic 3-manifold with finitely generated fundamental
group. Then, $$\lambda_0(N)=D(\Lambda(\Gamma))(2- D(\Lambda(\Gamma))$$
unless $D(\Lambda(\Gamma)<1$ in which case $\lambda_0(N)=1$.
\end{corollary}

{\bf Remark:}  The Hausdorff dimension of the limit set can only be less than 1 if
$\Gamma$ is a geometrically finite free group and it can only be equal to 1 if it
is a geometrically finite free group or a surface group which is conjugate to a subgroup of $\PSL_2(\R)$\
(see Braam \cite{braam}, Canary-Taylor \cite{canary-taylor} and Sullivan \cite{sullivan-flow}).

\medskip

Marc Culler, Peter Shalen and their co-authors have engaged in an extensive study of
the relationship between the topology and the volume of  a hyperbolic 3-manifold.
At the core of this study is a quantitative generalization of the Margulis lemma which
they originally established for purely hyperbolic, geometrically finite, free groups of rank two (and their limits) in \cite{culler-shalen-paradox} and generalized to free groups of all ranks in \cite{ACCS}.
The Tameness Theorem  (and the Density Theorem which we will discuss later)
allows us to remove the tameness  and the hyperbolicity assumptions.

\begin{theorem}{\rm (Anderson-Canary-Culler-Shalen \cite{ACCS})}
Let $\Gamma$ be a Kleinian group 
freely generated by elements
$\{\gamma_1,\ldots,\gamma_k\}$. If $z\in \H^3$, then
$$\sum_{i=1}^k {1\over{1+e^{d(z,\gamma_i(z))}}}\leq{1\over2}.$$
In particular there is some
$i\in\{1,\ldots,k\}$ such that
 $$d(z,\gamma_i(z))\geq\log(2k-1).$$
\end{theorem}

Culler and Shalen have an extensive body of work making use of the above estimate
to obtain volume estimates. 

In some cases, the Tameness Theorem yields immediate improvements of the
results in this program. For example, in the following result one originally also had
to assume that every 3-generator subgroup of $\pi_1(M)$ is topologically tame.

\begin{theorem}{\rm (Anderson-Canary-Culler-Shalen \cite{ACCS})}
If $N$ is a closed hyperbolic 3-manifold such that every 3-generator subgroup of
$\pi_1(N)$ is free, then the volume of $N$ is at least 3.08.
\end{theorem}

As another example of the type of results that Culler and Shalen obtain, we state one of their recent theorems, whose proof makes
use of the above estimate and the Tameness Theorem.

\begin{theorem}{\rm (Culler-Shalen \cite{culler-shalen-recent})}
If $N$ is a closed hyperbolic 3-manifold and $H_1(M,{\bf Z}_2)\ge 8$,
then the volume of $N$ is at least 3.08.
\end{theorem}

\section{Group-theoretic applications}

The resolution of Marden's Tameness Conjecture allows one to improve many previous
results concerning group-theoretic properties of hyperbolic 3-manifolds. The main
tool here is a corollary of the Covering Theorem which allows one to completely
characterize geometrically infinite covers of a finite volume hyperbolic 3-manifold.
The applications will be to the finitely generated intersection property, separability
properties of subgroups of Kleinian groups, the pro-normal topology on a Kleinian
group and to commensurators of subgroups of Kleinian groups.

The Covering Theorem produces restrictions on how a hyperbolic 3-manifold
with a simply degenerate end can cover another hyperbolic 3-manifold.
It was proved in the case that the covering manifold is geometrically tame
with incompressible boundary by Thurston \cite{thurston-notes} and in the case
where the covering manifold is allowed to be topologically tame by Canary \cite{canary-cover}.
(For versions where the base space is allowed to be an orbifold, see Agol \cite{agol}
and Canary-Leininger \cite{canary-leininger}.)

\medskip
\par\noindent
{\bf Covering Theorem:} {(Thurston\cite{thurston-notes}, Canary \cite{canary-cover})}
{\em Let $\hat N$ be  a 
hyperbolic 3-manifold with finitely generated fundamental group
which covers another hyperbolic 3-manifold $N$
by a local isometry $p : \hat N \to N$.
If $\hat E$ is a geometrically
infinite end of $\hat N$ then either

a) $\hat E$ has a neighborhood $\hat U$ such that $p$
is finite-to-one on $\hat U$, or

b) $N$  has finite volume and has a finite cover $N'$
which fibers over the circle such that if $N_S$ denotes the
cover of $N'$ associated to the fiber subgroup then $\hat N$ is
finitely covered by $N_S$.
Moreover, if $\hat N\ne N_S$, then
$\hat N$ is homeomorphic to the interior of a twisted I-bundle which is
doubly covered by $N_S$.}

\medskip

{\bf Remark:} The statement above assumes that $\hat N$ has no cusps. If $\hat N$ is
allowed to have cusps, then one must consider ends of $\hat N^0$, which is $N$ with
the cuspidal portions of its thin part removed. One can define simply degenerate
and geometrically infinite ends in this context and the statement is essentially the
same.

\medskip

In the case that $N=\H^3/\Gamma$ is a finite volume hyperbolic 3-manifold, we
see that all geometrically infinite, finitely generated subgroups of $\Gamma$ are associated
to fibre subgroups of finite covers of $N$ which fiber over the circle.
A subgroup $\hat \Gamma$ of $\Gamma$ is said to be  a
{\em virtual fiber subgroup}
if there exist finite index subgroups $\Gamma_0$ of $\Gamma$  and $\hat\Gamma_0$ of $\hat\Gamma$
such that
$N_0=\H^3/\Gamma_0$ fibers over the circle and $\hat\Gamma_0$
corresponds to the fiber subgroup. Corollary \ref{closedcovers} is the key tool in many of the group-theoretic applications of Marden's Tameness Conjecture.

\begin{corollary}
\label{closedcovers}
If $N=\H^3/\Gamma$ is a finite volume hyperbolic
3-manifold and $\hat\Gamma$ is a finitely generated
subgroup of $\Gamma$, then $\hat \Gamma$ is either geometrically finite
or a virtual fiber subgroup.
\end{corollary}

Thurston (see \cite{canary-tenn} for a proof) had earlier proved, using
Ahlfors' Finiteness Theorem \cite{ahlfors-AFT}, that a cover of an infinite-volume
geometrically finite hyperbolic 3-manifold is geometrically finite if it has finitely
generated fundamental group.
More generally, one may use the covering theorem
to completely describe exactly which covers of a hyperbolic 3-manifold with finitely generated
fundamental group are geometrically finite (see \cite{canary-cover}). 

Corollary \ref{closedcovers} is related to a question of Thurston.

\begin{question} 
\label{fiberquest}
{(Thurston)} {\em Does every finite volume hyperbolic 3-manifold
admit a finite cover which fibers over the circle?}
\end{question}

Ian Agol \cite{agol-virtual} has recently established that large classes of finite
volume hyperbolic 3-manifolds have finite covers which fiber over the circle. In particular, manifold
covers of reflection orbifolds and arithmetic hyperbolic orbifolds defined by a
quadratic form are covered by his methods.

\medskip

We first focus on the finitely generated intersection property.
A group $G$ is said to have the {\em finitely generated intersection
property} if whenever $H$ and $H'$ are finitely generated subgroups of
$G$, then $H\cap H'$ is finitely generated. Susskind \cite{susskind} proved that
the intersection of two geometrically finite subgroups of a Kleinian group is
geometrically finite. In combination with Thurston's proof that any finitely generated
subgroup of a co-infinite volume  geometrically finite Kleinian group is geometrically finite, this
establishes that co-infinite volume geometrically finite Kleinian groups have the finitely
generated intersection property. If we combine this with Thurston's Hyperbolization Theorem
for Haken 3-manifolds we obtain the following theorem of Hempel:

\begin{theorem}{\rm (Hempel \cite{hempel})}
Let $M$ be a compact, atoroidal, irreducible 3-manifold with
a non-toroidal boundary component. Then $\pi_1(M)$ has
the finitely generated intersection property.
\end{theorem}

It is well-known, see Jaco \cite{Jaco} for example, that hyperbolic  3-manifolds which fiber over
the circle do not have the finitely generated intersection property. However,
combining Susskind's result with Corollary \ref{closedcovers} we see that finite
volume hyperbolic 3-manifolds have the finitely generated intersection property if
and only if they do not have a finite cover which fibers over the circle.
Again, combining with the resolution of Thurston's Geometrization Conjecture, we get
the following purely topological statement.

\begin{theorem}
Let $M$ be a compact, atoroidal, irreducible 3-manifold whose fundamental group is
not virtually abelian.
Then $\pi_1(M)$ has
the finitely generated intersection property if and only if $M$ does not have a finite
cover which fibers over the circle.
\end{theorem}

{\bf Remark:} See Soma \cite{soma-fgip} for a discussion of the finitely generated
intersection property for geometric manifolds which are not hyperbolic.

\medskip

There are also a number of applications of Marden's Tameness Conjecture to
separability properties of fundamental groups of hyperbolic 3-manifolds.
If $G$ is a group, and $H$ a subgroup of $G$, then $H$ is said to be {\em separable} in $G$
if for every $g\in G\setminus H$, there is a subgroup $K$ of finite
index in $G$ such that $H \subset K$ but $g\notin K$.  $G$ is said to be
{\em LERF} if every finitely generated subgroup is separable.
This condition is a strengthening of residual finiteness
as a group is residually finite if and only if the trivial subgroup is separable. The main
motivation for studying this property comes from low-dimensional topology. If the fundamental
group of an irreducible 3-manifold  contains a separable surface subgroup, then one can find a finite
cover which contains an embedded incompressible surface and hence is Haken.

Scott \cite{scott-LERF,scott-LERF2} showed that all Seifert fibered manifolds have LERF fundamental
groups. However, Rubinstein and Wang \cite{rubinstein-wang} showed that there are graph manifolds whose fundamental
groups are not LERF (and in fact contain surface subgroups which are not separable.)

A Kleinian group $G$ is said to be GFERF if all geometrically finite subgroups are
separable. Since virtually fibered subgroups are easily seen to be separable, Corollary
\ref{closedcovers} implies that all GFERF Kleinian groups are in fact LERF. Many
examples of Kleinian groups are known to be GRERF, and hence LERF. Gitik \cite{gitik}
proved that fundamental groups of a large class of infinite volume hyperbolic
3-manifolds are LERF and also produced large classes of closed
hyperbolic 3-manifolds which are GFERF.
Wise \cite{wise} exhibited a class of non-postively curved 2-complexes
whose fundamental groups are separable for all quasiconvex subgroups. As one example,
he shows that the fundamental group of the figure eight knot complement (which is hyperbolic)
is GFERF. Agol, Long and Reid \cite{AGR} showed that the  
Bianchi groups ${\rm PSL}(2,{\mathcal{O}}_d)$ 
are GFERF. As a consequence they show that  there are infinitely many hyperbolic links in $S^3$
whose fundamental group are GFERF, e.g. the figure eight knot and the Borromean rings.
They also give examples of closed hyperbolic 3-manifolds whose fundamental groups are GFERF,
e.g. the Seifert-Weber dodecahedral space.

\medskip

Glasner, Souto and Storm \cite{GSS} applied the solution of Marden's Tameness Conjecture
and the Covering Theorem to the pro-normal topology on fundamental
groups of finite volume hyperbolic 3-manifolds.
We refer the reader to \cite{GSS} for the definition of the pro-normal topology
on the group. However, we recall that a subgroup is open if and only if it contains a non-trivial
normal subgroup and it is closed if and only if it is an intersection of open subgroups.
The pro-normal topology is not always defined, so they must also establish that the
pro-normal topology is well-defined.

\begin{theorem}{\em (Glasner-Souto-Storm\cite{GSS})}
\label{GSS}
Let $N$ be a finite volume hyperbolic 3-manifold.  The pro-normal topology on $\pi_1(N)$ is
well-defined and every finitely generated subgroup $H$ of $\pi_1(N)$ is closed in this topology.
Moreover, if $H$ has infinite index in $\pi_1(N)$, then it is the intersection of open subgroups
strictly containing $H$.
\end{theorem}

We say that a subgroup $H$ of a group $G$ is {\em maximal} if it is not strictly contained
in a proper subgroup of $G$.  It is clear, for example, that any subgroup of index 2 is maximal.
Margulis and Soifer \cite{margulis-soifer} proved that
every finitely generated linear group which is not virtually solvable contains a maximal
subgroup of infinite index.  It is natural to ask whether such a such a subgroup can be
can be finitely generated.

A subgroup is called {\em pro-dense}
if it is dense in the pro-normal topology and Gelander and Glasner \cite{gelander-glasner}
have established a result guaranteeing that fundamental groups of hyperbolic 3-manifolds
admit pro-dense subgroups. 

As a corollary of Theorem \ref{GSS}, Glasner, Souto and Storm showed that all
maximal subgroups of infinite index and all pro-dense subgroups are infinitely generated.

\begin{corollary}{\rm (Glasner-Souto-Storm\cite{GSS})}
Let $N$ be a finite volume hyperbolic 3-manifold.  If $H$ is a maximal subgroup of
infinite index or a pro-dense subgroup, then $H$ is infinitely generated.
\end{corollary}

{\bf Remark:} Their results also apply to fundamental groups of finite volume hyperbolic orbifolds.

\medskip

Alan Reid showed that one can use the Tameness Theorem and Corollary \ref{closedcovers}
to characterize when the commensurator of a subgroup of the fundamental group of
a finite volume hyperbolic 3-manifold has finite index in the entire group. We include
the proof of Reid's result with his kind permission.

We recall that if $H$ is a subgroup of a group $G$, then commensurator of $H$ in $G$
is defined to be
$$Comm_G(H)=\{g\in G\ | \ gHg^{-1} {\rm is\ commensurable\ with} \ H\}.$$
(We recall that $H$ and $J$ are {\em commensurable} subgroups of $G$ if $H\cap J$ has
finite index in both $H$ and $J$.)
Notice that if $G_0$ has finite index in $G$ , $H_0$ has finite index in $H$, and $H_0\subset G_0$,
then $Comm_{G_0}H_0\subset Comm_GH.$

\begin{theorem}{\rm (Reid)}
\label{reid}
Let $M$ be a finite volume hyperbolic 3-manifold and let $H$  be an infinite index,
finitely generated subgroup of $\pi_1(M)$. Then $Comm_{\pi_1(M)}H$ has finite
index in $\pi_1(M)$ if and only if $H$ is a virtual fiber subgroup.
\end{theorem}

\begin{proof}{}
First suppose that  $M=\H^3/\Gamma$ and $H$ is a virtual fiber subgroup of $\Gamma$.
Then there exist finite index subgroups $H_0\subset H$ and $\Gamma_0\subset \Gamma$
such that the cover $M_0$ of $M$ associated to $\Gamma_0$ is a finite volume hyperbolic
3-manifold which fibers over the circle and $H_0$ is the subgroup associated
to the fiber. Since $H_0$ is normal in $\Gamma_0$, 
$$\Gamma_0=Comm_{\Gamma_0}H_0\subset Comm_\Gamma H.$$

If $H$ is not a virtual fiber subgroup, then Corollary \ref{closedcovers} implies that $H$ is
geometrically finite. If $\gamma\in Comm_\Gamma H$, then $H\cap \gamma H\gamma^{-1}$
has finite index in both $H$ and $\gamma H\gamma^{-1}$. It follows that
$$\Lambda(H\cap \gamma H\gamma^{-1})=\Lambda(H)=\Lambda(\gamma H\gamma^{-1}).$$
Since $\Lambda(\gamma H\gamma^{-1})=\gamma(\Lambda(H))$, we see that $\Lambda(H)$
is invariant under $Comm_\Gamma H$, so $\Lambda(Comm_\Gamma H)\subset \Lambda(H)$.
However, since $H$ has infinite index in $\Gamma$, $\H^3/H$ has infinite volume and
$\Lambda(H)\ne \rs$. Since $\Lambda(\Gamma)=\rs$, $Comm_\Gamma H$ must have infinite
index in $\Gamma$.
\end{proof}

\medskip\noindent
{\bf Remark:} In fact, one can apply Ahlfors' Finiteness Theorem here to see that
if $H$ is not a virtual fiber subgroup, then $H$ has finite index in $Comm_\Gamma H$.

\section{Simon's Conjecture}

Marden's Tameness Conjecture is clearly closely related to Simon's
Conjecture on covers of compact 3-manifolds, which we state
in a slightly simpler form.

\begin{conjecture}{(Simon \cite{simon})}
{\em Let $M$ be a compact irreducible 3-manifold and let $N$ be a cover of $M$ with
finitely generated fundamental group, then ${\rm int}(N)$ is topologically tame.}
\end{conjecture}

Simon \cite{simon} originally proved his conjecture for covers associated to
peripheral subgroups of the 3-manifold group (i.e. subgroups of the image of
the fundamental group of a boundary component in the 3-manifold.)
 If one combines  Thurston's proof that
all covers (with finitely generated fundamental group) of a geometrically finite
infinite volume hyperbolic 3-manifold are geometrically finite 
with his Hyperbolization
Theorem for Haken manifolds, then one verifies Simon's conjecture for all compact, irreducible, atoroidal
3-manifolds with a non-toroidal boundary component.

The solution of Marden's Tameness Conjecture obviously implies Simon's conjecture
for all compact hyperbolizable 3-manifolds. Long and Reid pointed out that one can
combine this with Simon's work to establish Simon's conjecture for all manifolds
which admit a geometric decomposition. Since their argument has not appeared we
will give it here.

\medskip
\noindent{\bf Long and Reid's proof that Marden's Conjecture implies
Simon's conjecture for manifolds which admit a geometric decomposition:}
We first suppose that $M$ itself is a geometric manifold. If $M$ is hyperbolic, then
Simon's Conjecture follows directly from the Tameness Theorem.
If $M$ has SOL geometry, then $M$ has the finitely generated
intersection property (see Soma \cite{soma-fgip}), so
Simon's Conjecture follows from Theorem 3.7  in Simon \cite{simon}.
 In the remaining cases, $M$ is Seifert-fibered.
We may clearly assume that $\pi_1(M)$ is infinite. If $M$ admits a $S^2\times\R$-structure, then
it is finitely covered by $S^2\times S^1$ and Simon's Conjecture is easily verified.
In all other cases, $M$ contains an immersed incompressible torus. 

Since $\pi_1(M)$ is LERF, we can find a finite cover $\hat M$ of $M$
which containes an embedded non-separating torus
(see Scott \cite{scott-LERF,scott-LERF2}.)
Thus,  $\hat M$ is a union of $T^2\times I$ and $M_0$, where 
$M_0$ 
is a compact Seifert fibered space with non-empty boundary. Simon's conjecture 
holds for $M_0$, by Corollary 3.3 of \cite{simon}, since it has a finite cover  of the form
$F\times S^1$  where $F$ is a compact surface, and  it
clearly holds for $T^2\times I$.
Corollary 3.2 of \cite{simon}  then shows that Simon's conjecture holds for $\hat M$ and
hence for $M$.

Since we have established Simon's conjecture for all the pieces in a manifold which
admits a geometric decomposition, we can again apply Corollary 3.2 of
Simon \cite{simon} to complete the proof for compact irreducible manifolds
which admit a geometric decomposition.

\medskip

The recent resolution of Thurston's Geometrization Conjecture then
allows us to conclude that Simon's Conjecture holds for all compact irreducible 3-manifolds.

\begin{theorem}
Let $M$ be a compact irreducible 3-manifold and let $N$ be a cover of $M$ with
finitely generated fundamental group, then ${\rm int}(N)$ is topologically tame.
\end{theorem}

\section{Classification of hyperbolic 3-manifolds}

In 2003, Brock, Canary and Minsky \cite{ELC1,ELC2,ELC3} announced the proof
of Thurston's Ending Lamination Conjecture for  topologically tame hyperbolic 3-manifolds.
This conjecture gives a complete classification up to isometry of topologically
tame hyperbolic 3-manifolds in terms of their homeomorphism types and their
ending invariants (which encode the asymptotic geometry of their ends.)
The Tameness Theorem thus implies that we have a complete
classification of all hyperbolic 3-manifolds with finitely generated fundamental group.
A complete discussion of this classification theorem is outside the scope of this paper
but we will try to give a rough idea of the result. (See Minsky's survey paper \cite{minsky-harvard}
for
a more intensive exposition of the classificatiion theorem.)

Let $N$ be a hyperbolic 3-manifold with finitely generated fundamental group
and let $C$ be a compact core for $N$. We may assume that $N-C$ is homeomorphic
to $\partial C\times (0,\infty)$.
For simplicity, we will assume in our discussion that $N$ has no cusps.

If a component $S$ of $\partial C$ abuts a geometrically finite end of $N$, then
it inherits a conformal structure from the corresponding (homeomorphic) component $\bar S$ of
the conformal boundary and this conformal structure is the ending invariant
associated to that end.  If $U$ is the component of $N-C(N)$ bounded by $\bar S$, then there
is a natural map $r:U\to \partial C(N)$ given by taking a point in $U$ to the nearest point
on $\partial C(N)$. If $x\in U$, then every point on the geodesic ray beginning at $r(x)$ and passing
through $x$ is taken to $r(x)$, so this map induces a product structure on $U$.
One can check, see Epstein-Marden \cite{epstein-marden}, that $U$ is homeomorphic
to $\bar S\times (0,\infty)$ and that the metric on $U$ is bilipschitz to $\cosh^2(t)ds^2 +dt^2$
where $t$ is the real coordinate and $ds^2$ is the Poincar\'e metric on $\bar S$. So, one
sees very precisely that the conformal structure on $\bar S$ encodes the geometry of
the associated geometrically infinite end.

If a component $S$ of $\partial C$ abuts a simply degenerate end of $N$, then
there exists a sequence of simplicial hyperbolic surface $\{f_n:S\to N\}$ exiting the end.
There exists a uniform constant $B$, 
such that any metric induced on $S$ by a simplicial hyperbolic surface contains a simple closed
geodesic of length at most $B$.  For all $n$, let $\alpha_n$ be a simple closed curve on $S$ which
has length at most $B$ in the metric on $S$ induced by $f_n$. In this situation, the
{\em ending lamination} of the end associated to $S$ is the ``limit''  $\lambda$ of the
sequence $\{\alpha_n\}$.

We may make sense of the limit in two equivalent ways. In the first method, we fix
a hyperbolic structure on $S$ and let $\lambda$ be the Hausdorff limit of the
sequence $\{\alpha_n^*\}$ of geodesic representatives of the $\alpha_n$.
This limit is a closed set which is a disjoint union of simple geodesics, i.e.
a geodesic lamination.
(To be more precise, we must remove any isolated leaves in the resulting lamination,
to ensure that it is well-defined.)
The second method involves considering the curve complex $\mathcal{C}(S)$.
The vertices of the curve complex are isotopy classes of simple closed curves
and we say that a collection of vertices spans a simplex if and only if they
have a collection of disjoint representatives. The curve complex is locally infinite,
but Masur and Minsky \cite{masur-minsky} proved that it is Gromov hyperbolic.
We may then define $\lambda$ to be the point in the Gromov boundary which
is the limit of the vertices associated to the sequence $\{\alpha_n\}$.

Bonahon \cite{bonahon} and Thurston \cite{thurston-notes} proved that the
ending lamination is well-defined for simply degenerate ends of hyperbolic
3-manifold with finitely generated, freely indecomposable fundamental group
and Canary \cite{canary-ends} showed that they can be defined for simply degenerate
ends of topologically tame hyperbolic 3-manifolds. Klarreich \cite{klarreich},
see also Hamenstadt \cite{hamenstadt}, proved that one can identify the Gromov
boundary of the curve complex with the set of potential ending laminations.

The ending invariants of $N$ are encoded by the compact core $C$ where each
boundary component of $C$ is equipped with either a conformal structure or
an ending lamination. Thurston conjectured that this information determined
$N$ up to isometry. Brock, Canary and Minsky \cite{ELC1,ELC2,ELC3} proved
this conjecture for topologically tame hyperbolic 3-manifolds in a proof which
builds on earlier work of Masur and Minsky \cite{masur-minsky, masur-minsky2}.
The resolution of Marden's Tameness Conjecture gives the following:

\medskip\noindent
{\bf Ending Lamination Theorem:}
{\em A hyperbolic 3-manifold with finitely generated fundamental group is determined up to isometry by its ending invariants.}

\medskip

Alternate approaches to this result are given by Bowditch \cite{bowditch-ELC},
Brock-Bromberg-Evans-Souto \cite{BBES-ELC}, and Rees \cite{rees}. Minsky
\cite{minsky-PT} earlier established the Ending Lamination Theorem for
punctured torus groups.

We remark that one can determine exactly which end invariants arise, so
this is a complete classification theorem(see Ohshika \cite{ohshika-char} for the
characterization in the case that the fundamental group is freely indecomposable.)

One consequence of the proof of the Ending Lamination Theorem is the following 
common generalization of Mostow \cite{mostow} and Sullivan's \cite{sullivan-flow} rigidity theorems.

\begin{corollary}
Let $G$ be a finitely generated, torsion-free group which is not virtually abelian.
If two discrete faithful representations $\rho_1:G\to \PSL_2(\C)$ and $\rho_2:G\to\PSL_2(\C)$
are conjugate by an orientation-preserving
homeomorphism $\phi$ of $\rs$, then they are quasiconformally
conjugate. Moreover, if $\phi$ is conformal on $\Omega(\rho_1(G))$, then
$\phi$ is conformal.
\end{corollary}

\section{Deformation Theory of hyperbolic 3-manifolds}

It is natural to consider the space of all  (marked) hyperbolic 3-manifolds of fixed homotopy type.
One may think of this as a 3-dimensional generalization of Teichm\"uller theory.
The Mostow-Prasad Rigidity Theorem \cite{mostow,prasad} assures us that if a hyperbolic 3-manifold $N$ has finite volume, then any homotopy equivalence of $N$ to another hyperbolic 3-manifold
is homotopic to an isometry, so we will only consider the deformation theory of infinite
volume hyperbolic manifolds.

Let $M$ be a compact, atoroidal, irreducible 3-manifold with a non-toroidal boundary
component. We consider the space $AH(M)$ of (marked) hyperbolic 3-manifolds
homotopy equivalent to $M$. Formally, we define
$$AH(M) = \{\rho:\pi_1(M)\to {\rm PSL}_2({\bf C})|\ \  \rho\ 
{\rm discrete}\ {\rm and\ faithful }\}/{\rm PSL}_2({\bf C}).$$
The deformation space sits as a subset of the character variety
$$X(M)=Hom_T(\pi_1(M),{\rm PSL}_2({\bf C}))//{\rm PSL}_2({\bf C})$$
where $Hom_T(\pi_1(M),{\rm PSL}_2({\bf C}))$ denotes the space of
discrete faithful representations  $\rho:\pi_1(M)\to{\rm PSL}_2({\bf C})$ with
the property that if $g$ is a non-trivial element of a rank two abelian subgroup of
$\pi_1(M)$, then $\rho(g)$ is parabolic. An element $\rho\in AH(M)$ gives
rise to a pair $(N_\rho,h_\rho)$ where $N_\rho=\H^3/\rho(\pi_1(M))$ is a hyperbolic
3-manifold and $h_\rho:M\to N_\rho$ is a homotopy equivalence. (We could alternatively
have defined $AH(M)$ to be the set of such pairs up to appropriate equivalence.)

Marden \cite{marden} and Sullivan \cite{sullivan2} proved that the interior of $AH(M)$
(as a subset of $X(M)$) consists of geometrically finite representations such that
$\rho(g)$ is parabolic if and only if $g$ is a non-trivial element in a free abelian subgroup
of $\pi_1(M)$ of rank two. The classical deformation theory of Kleinian groups tells
us that  geometrically finite points in $AH(M)$ are determined by their homeomorphism
type and the conformal structures on their boundary (see Bers \cite{BersSurvey} or
Canary-McCullough \cite{canary-mccullough} for a survey of this theory.)

To state the parameterization theorem for $\iAH$ we need a few definitions.
We first define ${\mathcal A}(M)$ to be the set of oriented, compact, irreducible,
atoroidal (marked) 3-manifolds homotopy equivalent to $M$. More formally,
$\mathcal{A}(M)$ is the set of pairs $(M',h')$ where $M'$ is an oriented,
compact, irreducible, atoroidal 3-manifold and $h':M\to M'$ is a homotopy
equivalence where two pairs $(M_1,h_1)$ and $(M_2,h_2)$ are considered
equivalent if there exists an orientation-preserving homeomorphism $j:M_1\to M_2$
such that $j\circ h_1$ is homotopic to $h_2$.  We define $Mod_0(M')$ to be the group of isotopy
classes of homeomorphisms of $M'$ which are homotopic to the identity. We define
$\partial_{NT}M'$ to be the non-toroidal components of $\partial M$ and we let
${\mathcal T}(\partial_{NT}M')$ denote the Teichm\"uller space of all (marked) conformal
structures on $\partial_{NT}M'$.

\begin{theorem}{}{(Ahlfors, Bers, Kra, Marden, Maskit, Sullivan,Thurston)}{}
$$\iAH \cong \bigcup_{(M,h')'\in {\mathcal A}(M)} {\mathcal T}(\partial_{NT}M')/Mod_0(M')$$
\end{theorem}

In particular, we see that the components of $\iAH$ are in one-to-one correspondence
with elements of $\mathcal{A}(M)$. Moreover, Maskit \cite{maskit-free} showed that 
$Mod_0(M')$ always acts freely on ${\mathcal T}(\partial_{NT} M')$, so each component
is a manifold. Although, see McCullough \cite{mccullough-twist}, $Mod_0(M')$ is
often infinitely generated, so the fundamental group  of a component can be infinitely generated.
However, if $\pi_1(M)$ is freely indecomposable, then $Mod_0(M')$ is trivial for all
$(M',h')\in {\mathcal A}(M)$, so  in this case each component is topologically an open ball.

\medskip

Bers, Sullivan and Thurston conjectured that $AH(M)$ is the closure of its interior.
More concretely, this predicts that every hyperbolic 3-manifold with finitely generated
fundamental group is an (algebraic) limit of a sequence of geometrically finite
hyperbolic 3-manifolds. The Tameness Theorem is a crucial tool in the recent proof
of this conjecture.

\medskip\noindent
{\bf Density Theorem:} {\em  If $M$ is a compact hyperbolizable
3-manifold, then $AH(M)$ is the closure of its interior $\iAH$.}

\medskip

One may derive the Density Theorem from the Tameness
Theorem, the Ending Lamination Theorem,  and convergence results of
Thurston \cite{thurston2,thurston3}, Kleineidam-Souto \cite{kleineidam-souto},
Lecuire \cite{lecuire} and Kim-Lecuire-Ohshika \cite{KLO}.
Basically, the idea here is to consider the end invariants of a given 3-manifold,
use the convergence results to construct a hyperbolic 3-manifold
with the given end invariants which arises as a limit of geometrically finite
hyperbolic 3-manifolds. (In the case that the manifold is homotopy equivalent
to a compression body one must use clever arguments of Namazi-Souto \cite{namazi-souto}
or Ohshika \cite{ohshika-density} to verify that our limits have the correct
ending invariants.) One then applies the Ending Lamination Theorem to show
that our limit and our original manifold are the same.

The other approach makes use of the deformation theory of cone-manifolds
developed by Hodgson-Kerckhoff \cite{HK1,hodgson-kerckhoff} 
and Bromberg \cite{bromberg-deform}.  Many cases of the Density Theorem were
established by Bromberg and Brock-Bromberg \cite{brock-bromberg} and their approach
was generalized to prove the entire theorem by Bromberg
and Souto \cite{bromberg-souto}. This approach uses the Tameness Theorem but
not the Ending Lamination Theorem.

\medskip

Anderson, Canary and McCullough \cite{ACM} gave a complete enumeration
of the components of the closure of $\iAH$ in the case that $\pi_1(M)$ is
freely indecomposable. Given the resolution of the Density Conjecture, we now
have a complete enumeration of the components of $AH(M)$. 

Again, to state the result, we will need more definitions. 
Given two compact irreducible $3$-manifolds $M_1$ and $M_2$
with nonempty incompressible boundary, a homotopy equivalence
$h\colon\, M_1\to M_2$ is a {\em
primitive shuffle equivalence} if there exists a finite collection $V_1$ of
primitive solid torus components of $\Sigma(M_1)$ and a finite
collection $V_2$ of solid torus components of $\Sigma(M_2)$, so
that $h^{-1}(V_2)=V_1$ and so that $h$ restricts to an
orientation-preserving homeomorphism from the closure of
$M_1-V_1$ to the closure of $M_2-V_2$. We recall that a solid torus $V$  in
the characteristic submanifold is said to be {\em primitive} if 
given any annulus component $A$  of $V\cap \partial M$ 
the image of $\pi_1(A)$ in  $\pi_1(M)$
is a maximal abelian subgroup. 
(We refer the reader to
Jaco-Shalen \cite{jaco-shalen} or Johannson \cite{johannson} for a
discussion of the characteristic submanifold and to Canary-McCullough
\cite{canary-mccullough} for a discussion of the characteristic submanifold
in the setting of hyperbolizable 3-manifolds.)
Primitive shuffle equivalence induces a finite-to-one equivalence relation on
$\mathcal{A}(M)$ and we let $\widehat{\mathcal{A}}(M)$ denote the quotient.

Anderson, Canary and McCullough \cite{ACM} prove that if $\pi_1(M)$ is freely
indecomposable, then two components of
$\iAH$ have intersecting closure if and only if their associated (marked)
homeomorphism types differ by a primitive shuffle equivalence. Combining
this result with the Density Theorem we obtain:

\begin{theorem}
If $M$ is a compact, hyperbolizable 3-manifold with freely indecomposable
fundamental group, then the components of $AH(M)$
are in one-to-one correspondence with $\widehat{\mathcal{A}}(M)$.
\end{theorem}

Canary and McCullough \cite{canary-mccullough} proved that if $\pi_1(M)$ has
freely indecomposable fundamental group, then  $\mathcal{A}(M)$ has infinitely many
elements if and only if $M$ has double trouble.  $M$ is said to have {\em double trouble} if
there exist simple closed curves $\alpha$, $ \beta$ and $\gamma$ in $\partial M$ which
are homotopic in $M$, but not in $\partial M$, and $\alpha$ and $\beta$ lie on non-toroidal
boundary components of $M$, while $\gamma$ lies on a toroidal boundary component.
(Canary and McCullough \cite{canary-mccullough} give a complete analysis of 
when $\mathcal{A}(M)$ is finite in the general case.) So, we obtain the following corollary:

\begin{corollary}
Let $M$ be a compact, hyperbolizable 3-manifold with freely indecomposable
fundamental group. Then, $AH(M)$ has infinitely many components
if and only if $M$ has double trouble.
\end{corollary}

{\bf Remark:} The author has recently completed a survey article \cite{canary-india}
on the deformation theory of hyperbolic 3-manifolds which contains a more detailed
discussion of the topic.

\end{document}